\documentclass[twoside]{article}

\usepackage{amsmath,amsthm,amssymb}
\usepackage{newtxtext,newtxmath}

\usepackage[text={12.5cm,19cm},centering,paperwidth=17cm,paperheight=24cm]{geometry} 
\usepackage{epsfig, graphics}
\usepackage[fontsize=10.4pt]{scrextend}

\def\titlerunning#1{\gdef\titrun{#1}}

\def\hh{{\mathcal H}^{n-1}}
\def\rn{{{\mathbb R}^n}}

\makeatletter
\def\author#1{\gdef\autrun{\def\and{\unskip, }#1}\gdef\@author{#1}}
\def\address#1{{\def\and{\\\hspace*{15.6pt}}\renewcommand{\thefootnote}{}\footnote{#1}}\markboth{\autrun}{\titrun}}
\makeatother

\def\email#1{email: \href{mailto:#1}{#1} }

\newenvironment{dedication}{\itshape\center}{\par\medskip}
\newenvironment{acknowledgments}{\bigskip\small\noindent\textit{Acknowledgments.}}{\par}

\newtheorem{thm}{Theorem}[section]

\newtheorem{prob}[thm]{Problem}

\newtheorem{mainthm}[thm]{Main Theorem}

\theoremstyle{definition}
\newtheorem{defin}[thm]{Definition}

\newtheorem*{rem}{Remark}

\numberwithin{equation}{section}

\usepackage[hyperfootnotes=false,colorlinks=true,allcolors=blue]{hyperref}

\begin{document}

\titlerunning{Friedrichs  inequalities}

\title{\textbf{Friedrichs type inequalities \\ in arbitrary domains}}

\author{Andrea Cianchi \and Vladimir Maz'ya}

\date{}

\maketitle

\address{A. Cianchi:     Dipartimento di Matematica e Informatica \lq\lq U.Dini", Universit\`a di Firenze, Viale Morgagni 67/a 50134, Firenze (Italy); \email{andrea.ciamchi@unifi.it} \and V. Maz'ya: 
 Department of Mathematics, Link\"oping University, SE-581 83
Link\"oping, Sweden, and 
  Department of Mathematical Sciences, M\&O Building, 
University of Liverpool, Liverpool L69 3BX
 (UK), and 
 RUDN University, 
6 Miklukho-Maklay St, Moscow, 117198 (Russia);
\email{vladimir.mazya@liu.se}}

\begin{dedication}
To Ari Laptev on the occasion of his 70th birthday
\end{dedication}


\section{Introduction}\label{S:introd}

Almost one century ago, in his paper  \cite{Friedrichs}  on boundary value and eigenvalue problems in the elasticity theory of plates,  K.Friedrichs showed that, given  a  bounded open set  $\Omega \subset \mathbb R^2$ satisfying a very mild regularity condition on the boundary $\partial \Omega$, there exists a constant $C=C(\Omega)$ such that
\begin{equation}\label{friedineq}
\|u\|_{L^2(\Omega)} \leq C\big(\|\nabla u\|_{L^2(\Omega)} 
+ \|u\|_{L^2(\partial \Omega)}\big)
\end{equation}
for every sufficiently smooth function $u: \overline \Omega \to \mathbb R$. Here, $ \|\cdot\|_{L^2(\partial \Omega)}$ stands for the $L^2$-norm on $\partial \Omega$ with respect to the $(n-1)$-dimensional Hausdorff measure. As noticed in the same paper, an application of inequality \eqref{friedineq} to the first-order derivatives of $u$ yields 
\begin{equation}\label{friedineq2}
\|\nabla u\|_{L^2(\Omega)} \leq C\big(\|\nabla^2 u\|_{L^2(\Omega)} 
+ \|\nabla u\|_{L^2(\partial \Omega)}\big)
\end{equation}
for some constant $C=C(\Omega)$. 

Inequality \eqref{friedineq} can be regarded as a forerunner of inequalities in Sobolev type spaces involving trace norms over the boundary of the underlying domain.  The latter play a role in a variety of problems in the theory of partial differential equations, including
 the analysis of solutions to elliptic equations subject to Robin boundary conditions. Criteria for the validity of inequalities of the form
\begin{equation}\label{mazyaineq}
\|u\|_{L^q(\Omega)} \leq C_1\|\nabla u\|_{L^p(\Omega)} + C_2
 \|u\|_{L^r(\partial \Omega)},
\end{equation}
for some exponents $p,q,r \geq 1$,  in terms of isoperimetric or isocapacitary type inequalities relative to $\Omega$, can be found in \cite[Section 6.11]{Mabook}. As pointed out in the same section, from those criteria   necessary and sufficient conditions on $\Omega$ follow for the solvability of the problem
\begin{equation}\label{mazyaeq}
\begin{cases}
- {\rm div}(A(x) \nabla u) = 0 & \quad \text{in $\Omega$}
\\ A(x)\nabla u \cdot \nu + \lambda u = \varphi   & \quad \text{on $\partial \Omega$.}
\end{cases}
\end{equation}
Here, $A$ is an $n\times n$-matrix-valued function in $L^\infty(\Omega)$, such that $A(x)$ is symmetric and positive definite, with lowest eigenvalue unformly bounded away from $0$, the function $\varphi \in L^\sigma (\partial \Omega)$ for some $\sigma \in (1, 2]$, $\nu$ is the  normal outward unit vector on $\partial \Omega$ and $\lambda >0$. Criteria for the discreteness of the spectrum of the operator in the first equation in \eqref{mazyaeq} are also provided.

Remarkably,  no   regularity on $\Omega$ is needed for inequality \eqref{mazyaineq} to hold,  provided  that the exponents $p,q,r \geq 1$ are suitably related. If $1\leq p<n$, $r= \frac{p(n-1)}{n-p}$ and $q=\frac {pn}{n-p}$, then inequality  \eqref{mazyaineq} holds for   any open set $\Omega \subset\rn$, with constants $C_1$ and $C_2$ independent  of $\Omega$. In this case, the optimal constants $C_1$ and $C_2$ are also known. They appeared in 1960  in the paper \cite{Ma1960} for $p=1$. In this case the inequality reads
   \begin{equation}\label{1960} 
\|u\|_{L^{\frac n{n-1}}(\Omega)} \leq  \frac{\Gamma (1+\frac n2)^{\frac 1n}}{n \sqrt \pi}\Big(\|\nabla u\|_{L^1(\Omega)}  +
 \|u\|_{L^1(\partial \Omega)}\Big).
\end{equation}
 Inequality \eqref{1960} is in fact equivalent to the classical isoperimetric  inequality in $\rn$, as shown in \cite{Ma1960}. An application of this inequality to a suitable power   of $|u|$ easily implies a parallel  inequality for $1<p<n$ .
The optimal constants for these values of $p$ have much more recently been exhibited 
in \cite{MaggiVillani1}, via mass-transportation techniques. More generally, inequality  \eqref{mazyaineq} holds for any open set $\Omega$ with finite Lebesgue measure $\mathcal L^n(\Omega)$, provided that  $1\leq r \leq  \frac{p(n-1)}{n-p}$ and $q= \frac{rn}{n-1}$, the latter exponent being the largest possible if no additional assumption is imposed on $\Omega$ -- see \cite[Corollary 6.11.2 and subsequent Example]{Mabook}. Related results in the space of functions of bounded variation can be found in \cite{Rondi}.

In our paper \cite{CianchiMazya1} a theory of Friedrichs type inequalities, for any-order Sobolev spaces, in arbitrary domains is proposed, which allows for more general Borel measures in $\Omega$ on the left-hand sides and   more general norms of $u$ and of its derivatives. The idea pursued in that paper is that domain regularity can be replaced by appropriate information on boundary traces in Sobolev type inequalities. The punctum is that 
 the constants appearing in the inequalites in question are independent
of the geometry of $\Omega$. This is a  critical feature even when smooth domains are considered, and can be of use, for instance, in approximation arguments.

The present note centers on applications of results of \cite{CianchiMazya1} to first and second-order inequalities for specific classes of norms.  The dependence of the constants appearing in most of the relevant inequalities is also exhibited. Compact embeddings are proposed as well.

The first-order inequalities offered here have the form
\begin{equation}\label{firstineq}
\|u\|_{Y(\Omega, \mu)} \leq C_1\|\nabla u\|_{X(\Omega)} + C_2
 \|u\|_{Z(\partial \Omega)},
\end{equation}
where $\|\cdot\|_{X(\Omega)}$ and  $\|\cdot \|_{Z(\partial \Omega)}$ are Banach function norms on $\Omega$ and $\partial \Omega$ with respect to the Lebesgue measure and the $(n-1)$-dimensional Hausdorff measure, respectively, and  $\|\cdot\|_{Y(\Omega, \mu)}$ is a Banach function norm with respect to an $\alpha$-upper Ahlfors regular measure $\mu$, with $\alpha \in (n-1, n]$. The norms that will come into play in our discussion are of Lebesgue, Lorentz, Lorentz-Zygmund and Orlicz type.

The second-order inequalities to be presented disclose further new traits. A major novelty with respect to \eqref{friedineq2}, and to other customary inequalities, is that the boundary norms only depend on the trace of $u$ on $\partial \Omega$ and not on that of $\nabla u$. Indeed, our second-order inequalities for $u$ read
\begin{equation}\label{secondineq}
\|u\|_{Y(\Omega, \mu)} \leq C_1\|\nabla^2 u\|_{X(\Omega)} +  C_2 \|g_u\|_{U(\partial \Omega)}+ C_3
 \|u\|_{Z(\partial \Omega)},
\end{equation}
where $\|\cdot\|_{X(\Omega)}$, $\|\cdot \|_{Z(\partial \Omega)}$ and $\|\cdot\|_{Y(\Omega, \mu)}$ are as above,   $\|\cdot \|_{U(\partial \Omega)}$ is a Banach function norm on   $\partial \Omega$, and $g_u$ denotes any upper gradient, in the sense of Haj\l asz \cite{Hajlasz}, of the trace of $u$ over $\partial \Omega$. The function $g_u$ can be regarded as a surrogate for the tangential gradient of $u$ over $\partial \Omega$ when the latter is not smooth, and hence it is not endowed with a differential structure.
\\
Let us emphasize  that, although inequalities of the form \eqref{secondineq} are known to hold in regular domains even without the middle term on the right-hand side, such a term cannot be dispensed with in highly irregular domains, such as the one in the figure. 
We refer to the paper \cite{CianchiMazya1} (see Example 7.1) for an example with this regard, and for more examples demonstrating the sharpness of our results under various respects.

\begin{figure}
\begin{center}
\includegraphics[height=7cm]{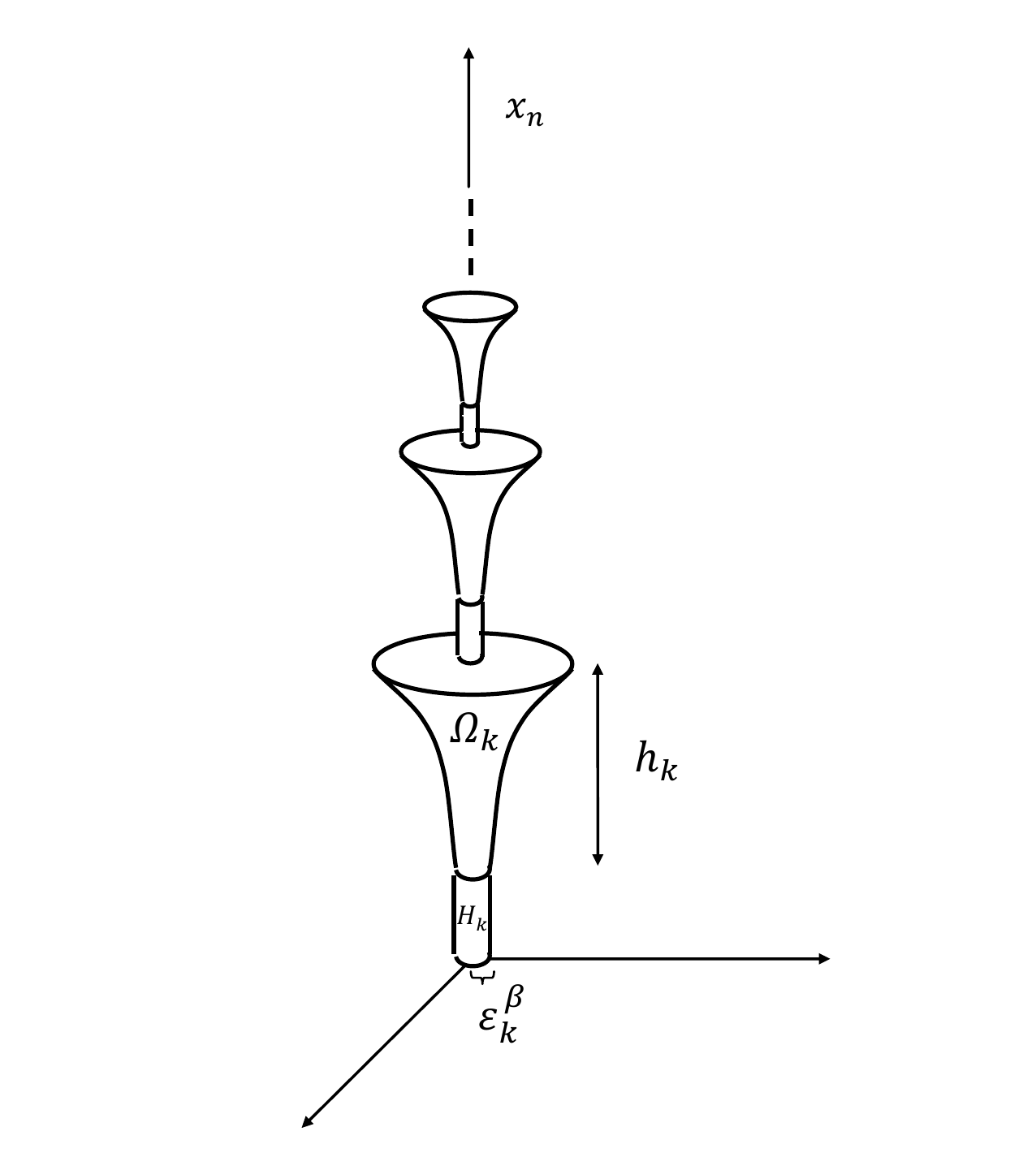}
\end{center}
        \label{Fig3}
\end{figure}

Bounds for first-order derivatives can also be provided in the form
\begin{equation}\label{secondineqgrad}
\|\nabla u\|_{Y(\Omega, \mu)} \leq C_1\|\nabla^2 u\|_{X(\Omega)} + C_2
 \|g_u\|_{U(\partial \Omega)}.
\end{equation}
This inequality will be shown to hold, in particular, when $\mu=\mathcal L^n$, $\Omega$ is such that $\mathcal L^n (\Omega)<\infty$, $\hh (\partial \Omega)<\infty$, and $X=Y=Z=L^2$.  Hence, it   augments inequality \eqref{friedineq2} in that it only involves the  boundary trace of $u$, through $g_u$, instead of the trace of $\nabla u$, on the right-hand side. 

Finally, we report on companion  inequalities from \cite{CianchiMazya2} for vector-valued functions ${\bf u} : \Omega \to \rn$, where the full matrix-valued gradient $\nabla {\bf u}$ is replaced by its sole symmetric part, denoted by $\mathcal E {\bf u}$ and called symmetric gradient of ${\bf u}$. Namely, 
$$\mathcal E {\bf u}= \tfrac 12\big(\nabla {\bf u} + (\nabla {\bf u})^T\big)\,,$$
where  $(\nabla{\bf u})^T$ stands for the transpose matrix of $\nabla {\bf u}$.
The symmetric gradient enters various
mathematical models for physical phenomena.
Instances in this connection are provided by the theory of non-Newtonian fluids, where the function ${\bf u}$ represents the velocity of a fluid  \cite{AsMa, DL, FuS,
MalNRR, MR, MaRa}, and the theories of plasticity and nonlinear elasticity, where ${\bf u}$ stands for the displacement of a body \cite{FuS, Kohn, Te}.  

The approach developed in \cite{CianchiMazya1} and \cite{CianchiMazya2} rests upon pointwise bounds for functions and their derivatives in terms of classical Riesz potentials of their highest-order derivatives, plus integral operators of a new kind. These bounds, and ensuing estimates in rearrangement form and reduction principles to one-dimensional inequalities, are reproduced here in the cases of use for our purposes.
To give an idea of how the reduction principles can be exploited, we also reproduce a sketch of the proofs of a couple of   results on second-order inequalities.

\section{Background and notations}\label{backgr}

Let $\mathcal R $ be a  measure space endowed with a
$\sigma$-finite, non-atomic measure $m$.   The decreasing rearrangement $\phi_m\sp*
: [0, \infty ) \to [0, \infty]$ of a $m$-measurable function $\phi
: \mathcal R \to \mathbb R$
 is defined as
$$
\phi_m\sp*(t)=\inf\{{\tau \geq 0}:\,m (\{|\phi|>\tau\})\leq t\}\quad
\textup{for}\ t\in [0,\infty).
$$
A
rearrangement-invariant space $X(\mathcal R, m)$ on 
$(\mathcal R, m)$ is a Banach
function space (in the sense of Luxemburg) equipped with a norm
$\|\cdot \|_{X(\mathcal R, m)}$ such that
\begin{equation}\label{ri1}
\|\phi\|_{X(\mathcal R, m)} = \|\psi\|_{X(\mathcal R, m)}\quad
\hbox{whenever} \quad \phi^*_m = \psi^*_m.
\end{equation}
Every rearrangement-invariant space $X(\mathcal R, m)$ admits a
representation space $\overline X(0, \infty)$, namely another
rearrangement-invariant space  on $(0, \infty)$, endowed with the Lebesgue measure, such that
\begin{equation}\label{ri2}
\|\phi\|_{X(\mathcal R, m)} = \|\phi^*_m\|_{\overline X(0,
\infty)}\quad \hbox{for every $\phi \in X(\mathcal R, m)$.}
\end{equation}

A basic example of rearrangement-invariant space is the Lebesgue space $L^p(\mathcal R, m)$, defined as usual for $p\in[1,\infty]$.

A generalization of the Lebesgue spaces is constituted by the  Lorentz spaces. Assume that either $1<p<\infty$
and $1\leq \sigma\leq\infty$, or $p=\sigma=1$, or $p=\sigma=\infty$. Then the functional given by
$$
\|\phi\|_{L\sp{p,\sigma}(\mathcal R, m)}=
\left\|t\sp{\frac{1}{p}-\frac{1}{\sigma}}\phi^*_m(t)\right\|_{L\sp \sigma(0,m (\mathcal R))}
$$
for  a $m$-measurable function $\phi$ on $\mathcal R$ is (equivalent) to a~rearrangement-invariant  norm. Here, and in what follows,  we use the convention that $\frac1{\infty}=0$. Observe that $L\sp{p,p}(\mathcal R, m)= L\sp{p}(\mathcal R, m)$ for $p \in [1,\infty]$. Moreover, $L\sp{p,\sigma_1}(\mathcal R, m)\subsetneq L\sp{p,\sigma_2}(\mathcal R, m)$ if $\sigma_1 < \sigma _2$.
\\
The Lorentz spaces admit, in their turn, a further extension, on finite measure spaces, provided by the Lorentz-Zygmund spaces. 
Assume  that $1\leq p,q\leq\infty$ and $\vartheta \in\mathbb R$, and that $m (\mathcal R)< \infty$. The Lorentz-Zygmund
functional is defined as 
\begin{equation}\label{E:1.18}
\|\phi \|_{L\sp{p,\sigma;\vartheta}(\mathcal R, m)}=
\left\|t\sp{\frac{1}{p}-\frac{1}{\sigma}}\log \sp
\vartheta \Big(1+\tfrac{m(\mathcal R)}{t}\Big) \phi^*_m(t)\right\|_{L\sp \sigma(0,m (\mathcal R))}
\end{equation}
for  a $m$-measurable function $\phi$ on $\mathcal R$. For suitable choices of the parameters  $p,\sigma, \theta$,  this functional
 is equivalent to rearrangement-invariant
norm.  A description of the admissible choices of these parameters  and of the main properties  of
Lorentz-Zygmund spaces can be found in 
\cite{EOP, glz}  and \cite[Chapter~9]{PKJF}.

The Orlicz spaces generalize the Lebesgue spaces in a different direction. Let $A: [0, \infty) \to
[0, \infty]$ be a Young function, namely  a left-continuous convex function which is
neither identically equal to $0$ nor to $\infty$. Then the Orlicz space 
$L^A(\mathcal R, m)$ associated with $A$ is the
rearrangement-invariant space equipped the Luxemburg norm given by
\begin{equation}\label{Orlicz}
\| \phi \|_{_{L^{A}(\mathcal R, m)}} =  \ \inf \Bigg\{ \lambda > 0 \ :
\ \int_{\mathcal R} A\bigg(\frac{|\phi |}{\lambda} \bigg)\, d\nu \
\leq \ 1 \Bigg\}.
\end{equation}

The Lebesgue space  $L^{p}(\mathcal R, m)$ is recovered with the choice   $A(t)=t^p$ if $p \in [1,
\infty)$, and  
$A(t)=\infty \chi_{_{(1, \infty)}}(t)$  if $p=\infty$, where $\chi_E$ denotes the characteristic function of a set $E$. If f $m(\mathcal R)< \infty$ and   $\gamma
>0$, we denote by $\exp L^\gamma (\mathcal R, m)$ the Orlicz space
built upon  a Young function $A(t)$ equivalent to $e^{t^\gamma } -1$ near infinity.
 Notice that  the space  $\exp L^\gamma (\mathcal R, m)$   belongs the family of Lorentz-Zygmund spaces as well. Indeed,  
\begin{equation}\label{equivexp}
\exp L^\gamma (\mathcal R, m) = L^{\infty, \infty;-\frac 1\gamma}(\mathcal R, m),
\end{equation}
up to equivalent norms, depending on $\gamma$ and on $m(\mathcal R)$.

Throughout this paper, $\mathcal R$ will be either an open subset $\Omega$ of $\rn$, or its topological boundary $\partial \Omega$. The Lebesgue measure in $\rn$ is denoted by $\mathcal L^n$. The notation $\hh$ is employed for the $(n-1)$-dimensional Hausdorff measure restricted to $\partial \Omega$. The  measure will be omitted in the notation of rearrangement-invariant spaces on $\Omega$ or on $\partial \Omega$ equipped with these standard measures. 

Upper Ahlfors regular measures, also called Frostman measures, on $\Omega$ will also considered. 
Recall that  an $\alpha$-upper Ahlfors regular measure 
$\mu$ on $\Omega$,  with  $\alpha \in (0, n]$,  is a Borel measure  
such that
\begin{equation}\label{measure}
\mu (B_r(x) \cap \Omega ) \leq C_\mu r^\alpha\quad \hbox{for $x \in
\Omega$ and  $r>0$,}
\end{equation}
for some   constant $C_\mu>0$. Here,
$B_r(x)$ denotes the ball, centered at $x$, with radius $r$.

The spaces of continuous functions in $\Omega$ and in $\overline \Omega$  are denoted by $C(\Omega)$ and $C(\overline \Omega)$, respectively. The notation $C_b(\overline \Omega)$ is adopted for the space of functions in $C(\overline \Omega)$ with bounded support. The restriction to $\partial \Omega$ of a function $u \in C(\overline \Omega)$ is denoted by $u_{\partial \Omega}$. We shall simply write $u$ instead of  $u_{\partial \Omega}$ in the notation of norms of  $u_{\partial \Omega}$ over $\partial \Omega$.

For each $x \in
 \Omega$,
 we set
\begin{equation}\label{partomegax}
(\partial \Omega )_x = \{y \in \partial \Omega: (1-t)x + ty \subset
\Omega\,\,\, \hbox{for every $t \in (0,1)$}\}.
\end{equation}
This is the largest subset  $\partial \Omega $ which can be \lq\lq seen" from $x$.
One can show that   $(\partial
\Omega )_x$ is a Borel
set.
Next, we define the set
\begin{equation}\label{finite}
(\Omega \times \mathbb S ^{n-1})_0 = \{(x, \vartheta ) \in \Omega
\times \mathbb S ^{n-1}: x+t\vartheta \in \partial \Omega\,\,
\hbox{for some $t>0$}\}.
\end{equation}
Clearly, $(\Omega \times \mathbb S ^{n-1})_0= \Omega \times \mathbb S ^{n-1}$ if $\Omega $ is bounded.
 Let $\zeta  : (\Omega \times \mathbb S ^{n-1})_0
\to \rn$
 be the function defined as
$$\zeta (x, \vartheta) = x + t \vartheta, \quad \hbox{where $t$ is
such that $x + t \vartheta \in (\partial \Omega) _x$}.$$
  In other
words, $\zeta (x, \vartheta)$ is the first point of intersection of
the half-line $\{ x + t \vartheta: t >0\}$ with $\partial \Omega$.
\par
 Given a function $\varphi : \partial \Omega \to \mathbb R$, with bounded
 support, we adopt the convention that  $g(\zeta (x, \vartheta))$ is defined for every
 $(x, \vartheta) \in \Omega \times \mathbb S ^{n-1}$, on extending it by $0$ on  $(\Omega \times \mathbb S ^{n-1}) \setminus (\Omega \times \mathbb S ^{n-1})_0$; namely, we set
\begin{equation}\label{conv}
\varphi (\zeta (x, \vartheta))= 0 \quad  \hbox{if $(x, \vartheta) \in
(\Omega \times \mathbb S ^{n-1}) \setminus (\Omega \times \mathbb S ^{n-1})_0$.}
\end{equation}

\section{First-order   inequalities}

In this section we are concerned with inequalities for first-order Sobolev spaces.

Given  a rearrangement-invariant space $X(\Omega)$ on an open set $\Omega \subset \rn$, we denote by $V^{1}X(\Omega)$
 the first-order homogeneous Sobolev type space defined as
\begin{equation}\label{sobolevV}
V^{1}X (\Omega ) =   \big\{u: \hbox{$u$ is   weakly
differentiable in $\Omega$, and $|\nabla u| \in
X(\Omega)$}\big\}.
\end{equation}
 Notice that no assumption on the integrability of $u$ is made  in the definition of $V^{1}X (\Omega )$.

In  
 the following statements, norms of Sobolev functions and of their derivatives with respect to an $\alpha$-upper Ahlfors regular measures $\mu$ appear. The relevant functions have to interpreted in the sense of traces with respect to $\mu$. These traces are well defined, thanks to standard (local) Sobolev
inequalities with measures, owing to the assumption that $\alpha \in
(n-1, n]$ in \eqref{measure}. An analogous convention applies to the
integral operators that enter our discussion.

\subsection{First-order Friedrichs type inequalties}

We begin with inequalities involving classical Sobolev and boundary trace spaces built upon Lebesgue norms. The target spaces are also of Lebesgue type,   except for a borderline case, where Orlicz spaces of exponential type naturally come  into play.

\begin{thm}\label{fried_p<n} {\bf [First-order   inequalities]}
Let $\Omega$ be any open  set in $\rn$, $n \geq 2$ with $\mathcal L^n(\Omega)<\infty$ and $\hh (\Omega)< \infty$.  Assume that
$\mu$ is a   Borel measure in $\Omega$ fulfilling \eqref{measure} for
some $\alpha \in (n-1, n]$ and for some $C_\mu
>0$, and such that $ \mu (\Omega)< \infty$. 
\\ (i) 
Let $1<p<n$,  $r>1$, and 
\begin{equation} \label{q} q= \min \bigg\{\frac {p\alpha }{n-p}, \frac {r\alpha }{n-1}\bigg\}.
\end{equation}
Then there exists a constant 
$C=C(n, p, r, \alpha, C_\mu)$ such that
\begin{align}\label{fried1}
\|u\|_{L^q(\Omega, \mu)}& \leq C \Big(\max\{\mu(\Omega)^{\frac n\alpha}, \mathcal L^n(\Omega)\}   ^{\frac \alpha{qn}-\frac{n-p}{pn}} \|\nabla u\|_{L^p(\Omega)}  \nonumber
\\  & \qquad + \max\{\mu(\Omega)^{\frac {n-1}\alpha}, \hh(\partial \Omega)\}  ^{\frac \alpha{q(n-1)}-\frac{1}{r}} \|u\|_{L^r(\partial \Omega)}\Big)
\end{align}
for every $u \in V^{1}L^p(\Omega ) \cap C_b(\overline \Omega )$. 
\\ (ii)
Let $\beta>0$,  and 
\begin{equation}\label{gamma}
\gamma= \min\{n', \beta\}.
\end{equation}
Then there exists a constant 
$C=C(n,  	\beta, \alpha, C_\mu, \mathcal L^n (\Omega), \hh (\Omega), \mu(\Omega))$ such that
\begin{align}\label{fried3}
\|u\|_{\exp L^\gamma (\Omega, \mu)} \leq C \Big( \|\nabla u\|_{L^n(\Omega)} +  \|u\|_{\exp L^\beta(\partial \Omega)}\Big)
\end{align}
for every $u \in V^{1}L^n(\Omega ) \cap C_b(\overline \Omega )$. 
\\ (iii)
 Let $p
>n$. Then
 there
 exists a  constant $C=C(n, p)$
  such that
\begin{align}\label{infdisp}
\|u\|_{L^\infty (\Omega )} & \leq C \Big(\mathcal L^n(\Omega)^{\frac 1n-\frac 1p}\|\nabla u\|_{L^{p}(\Omega )}
 +     \|u\|_{L^\infty(\partial \Omega)}\Big)
\end{align}
for every $u \in V^{1}L^p(\Omega) \cap C_b(\overline \Omega)$.
\end{thm}

 \begin{rem}
Clearly, the inequalities of Theorem  \ref{fried_p<n}  continue to hold for functions in the closure of the space $V^{1}L^p(\Omega) \cap C_b(\overline \Omega)$ with respect to the norms appearing on their right-hand sides. An analogous remark holds for all our inequalities.
 \end{rem}

  \begin{rem}
If $q=\frac {p\alpha }{n-p}$, then the exponent of the coefficient multiplying the gradient norm in inequality \eqref{fried1} vanishes, and, in fact, the assumption $\mathcal L^n(\Omega)< \infty$ can be dropped. On the other hand, if
$q=\frac {r\alpha }{n-1}$, then  the exponent of the coefficient  multiplying the boundary norm in inequality \eqref{fried1} vanishes, and the assumption $\hh (\partial \Omega) <\infty$ can be removed.  Analogous weakenings of the assumptions on $\Omega$ are admissible in the inequalities stated below, whenever the constants involved turn out to be independent of   $\mathcal L^n(\Omega)$ or $ \hh (\partial \Omega)$.
\end{rem}

Part (i) of Theorem \ref{fried_p<n} 
 extends a version  of the  Sobolev inequality for
measures, on regular domains \cite[Theorem 1.4.5]{Mabook}. It also
augments the results mentioned in Section \ref{S:introd} on inequality \eqref{mazyaineq} about general domains, 
which are  confined to norms on the left-hand side  with respect to the
Lebesgue measure. 
 Part (ii)    generalizes  the
 Yudovich-Pohozaev-Trudinger  inequality to possibly irregular
 domains. Moreover, it     enhances, under some respect, a result of
 \cite{MaggiVillani2}, where optimal constants are exhibited, but  for a weaker exponential norm,  and just  with the Lebesgue measure, on the left-hand side.

\smallskip
The following theorem provides us with a compactness result for subcritical norms on the left-hand side.

\begin{thm}\label{compactp<n} {\bf [Compact   embeddings]}
Let $\Omega$ be any open  set in $\rn$, $n \geq 2$ with $\mathcal L^n(\Omega)<\infty$ and $\hh (\Omega)< \infty$.  Let
$\mu$ be a   Borel measure in $\Omega$ fulfilling \eqref{measure} for
some $\alpha \in (n-1, n]$ and for some $C_\mu
>0$, and such that $ \mu (\Omega)< \infty$.  
\\ (i) Let $1<p<n$ and $r>1$.
Assume that
\begin{equation}\label{compnew1}
1 \leq q  < \min \bigg\{\frac {p\alpha }{n-p}, \frac {r\alpha }{n-1}\bigg\}.
\end{equation}
If
 $\{u_k\}$
is a bounded sequence in the space $V^{1}L^p(\Omega ) \cap
 C_{\rm b}(\overline \Omega )$ endowed with the norm appearing
on the right-hand side of inequality \eqref{fried1},  then $\{
u_k\}$ is a Cauchy sequence in $L^q(\Omega, \mu)$.
\\ (ii)  Assume that
\begin{equation}\label{compnew2}
0<\gamma <\min\{n', \beta\}
\end{equation}
If
 $\{u_k\}$
is a bounded sequence in the space $V^{1,n}(\Omega ) \cap
 C_{\rm b}(\overline \Omega )$ endowed with the norm appearing
on the right-hand side of inequality \eqref{fried3},  then $\{
u_k\}$ is a Cauchy sequence in $\exp L^{\theta}(\Omega, \mu)$.
\end{thm}

The next  result concerns inequalities for functions whose  gradient belongs to a Lorentz space $L^{p,\sigma}(\Omega)$. It extends  and improves Theorem \ref{fried_p<n}, since $L^{p,p}(\Omega) = L^p(\Omega)$.  In fact, it  augments the conclusions of   Parts (i) and (ii) of Theorem \ref{fried_p<n} also in this special case when $\sigma =p$, inasmuch as it provides us with strictly smaller  target spaces in the respective inequalities.

\begin{thm}\label{Lorentz} {\bf [First-order Lorentz--Sobolev inequalities]}Let $\Omega$ be any open  set in $\rn$, $n \geq 2$ with $\mathcal L^n(\Omega)<\infty$ and $\hh (\Omega)< \infty$.  Let
$\mu$ be a   Borel measure in $\Omega$ fulfilling \eqref{measure} for
some $\alpha \in (n-1, n]$ and for some $C_\mu
>0$, and such that $ \mu (\Omega)< \infty$.  
\\ (i) Assume  that $1 < p <
n$, $r>1$, and $1\leq \sigma, \varrho \leq \infty$. Let  $q$ be defined as in \eqref{q} and let
$$\eta \geq 
\begin{cases}
\sigma & \quad \text{if $q =\frac {p\alpha }{n-p}$}
\\
\varrho  & \quad \text{if $q=\frac {r\alpha }{n-1}$.}
\end{cases}
$$
Then
 there
 exists a constant $C=C(n, p, r,  \sigma, \varrho,  \alpha,  C_\mu)$ such that
\begin{align}\label{mainlor}
\|u\|_{L^{q,\eta}(\Omega, \mu)} & \leq
C \Big(\max\{\mu(\Omega)^{\frac n\alpha}, \mathcal L^n(\Omega)\}   ^{\frac \alpha{qn}-\frac{n-p}{pn}}\|\nabla  u\|_{L^{p,\sigma}(\Omega )} \nonumber
\\  & \qquad  +   \max\{\mu(\Omega)^{\frac {n-1}\alpha}, \hh(\partial \Omega)\}  ^{\frac \alpha{q(n-1)}-\frac{1}{r}}
\|u\|_{L^{r, \varrho}(\partial \Omega )}\Big)
\end{align}
for every $u \in V^{1}L^{p,q}(\Omega) \cap C_b(\overline \Omega)$.
\\ (ii) Assume that   $\sigma , \varrho >1$,  $\varsigma < - \tfrac 1\varrho$.    Let 
$$\eta \geq \max \{\sigma,   \varrho\},$$
and 
$$\vartheta \leq \min\Big\{-1+\tfrac 1 \sigma -\tfrac 1 \eta, \varsigma + \tfrac 1 \varrho -\tfrac 1\eta\Big\}.$$
 Then there exists a constant $C=C(n, \eta, \sigma, \varrho, \theta,  \varsigma, \alpha, C_\mu, \mathcal L^n (\Omega), \mu(\Omega), \hh (\partial \Omega))$ such that
 \begin{align}\label{BWlor}
\|u\|_{L^{\infty, \eta; \theta}(\Omega, \mu)}  \leq
C\Big( \|\nabla u\|_{L^{n,\sigma}(\Omega )}
  +
 \|u\|_{L^{\infty, \varrho; \varsigma}(\Omega)}\Big)
\end{align}
for every $u \in V^{1}L^{n,\sigma}(\Omega) \cap C_b(\overline \Omega )$.
\\ (iii) Assume that either $p=n$ and $\sigma =1$, or $p>n$ and $1\leq \sigma\leq \infty$.  
Then there
 exists a constant $C=C(n, p, \sigma, \alpha, C_\mu)$ such that
\begin{align}\label{inflor}
\|u\|_{L^\infty (\Omega )} & \leq C \Big(\mathcal L^n(\Omega)^{\frac 1n-\frac 1p}\|\nabla u\|_{L^{p,\sigma}(\Omega )}
 +     \|u\|_{L^\infty(\partial \Omega)}\Big)
\end{align}
for every $u \in V^{1}L^{p,\sigma}(\Omega) \cap C_b(\overline \Omega)$.
\end{thm}

Part (i) of Theorem \ref{Lorentz} is a counterpart in the present setting of 
classical results of \cite{Oneil, Peetre}.  The  borderline situation considered in Part (ii) corresponds to  an optimal integrability result which follows from a capacitary inequality of \cite{Ma1973} -- see \cite[Inequality (2.3.14)]{Mabook}.

\subsection{First-order poinwise estimates and reduction principle}\label{1est}

The point of departure of our method is the following pointwise estimate involving an unconventional integral operator.

\begin{thm}\label{point1}
{\bf [First-order pointwise estimate]} Let $\Omega$ be any  open set in $\rn$,
$n \geq 2$.  Then there exists a constant $C=C(n)$ such that
\begin{align}\label{point1.1}
|u(x)| & \leq C \Bigg(\int _\Omega \frac{|\nabla u(y)|}{|x-y|^{n-1}}\, dy +    \int _{\mathbb S^{n-1}}|u_{\partial \Omega}(\zeta (x, \vartheta ))|\,
d\hh (\vartheta )\Bigg) \qquad \hbox{for  $x \in \Omega$,}
\end{align}
for every  $u \in V^{1 }L^1(\Omega ) \cap C_{\rm
b}(\overline \Omega )$. Here,   convention
\eqref{conv} is adopted.
\end{thm}

Sharp endpoint continuity properties of the integral operators on the right-hand side of inequality \eqref{point1.1}, combined with interpolation techniques, yield the bound, in rearrangement form, stated in the next theorem. 

\begin{thm}\label{rearrest_1} {\bf [First-order rearrangement estimate]}
Let $\Omega$ be any open  set in $\rn$, $n \geq 2$.   Assume that
$\mu$ is a Borel measure in $\Omega$ fulfilling \eqref{measure} for
some $\alpha \in (n-1, n]$ and for some $C_\mu
>0$. Then there exist  constants $c=c(n)$ and $C=C(n,\alpha ,
C_\mu)$ such that
\begin{align}\label{rearr1}
u_\mu^*(ct)  & \leq C\Bigg(t^{-\frac
{n-1}{\alpha}}\int _0^{t^{\frac n\alpha}}|\nabla u|_{\mathcal
L^n}^*(\rho) d\rho + \int _{t^{\frac n\alpha}}^\infty \rho^{-\frac{n-1}n}
|\nabla u|_{\mathcal L^n}^*(\rho) d\rho \nonumber  \\  
& \quad \quad \quad + t^{-\frac {n-1}{\alpha}}\int _0^{t^{\frac {n-1}\alpha}}
\big(u_{\partial \Omega}\big)_{\hh}^*(\rho)d\rho\Bigg)\quad
\quad \hbox{for $t>0$,}
\end{align}
for every $u \in V^{1}L^1(\Omega ) \cap C_{\rm
b}(\overline \Omega )$.  
\end{thm}

The rearrangement estimate \eqref{rearr1} enables us to reduce inequalities of the form \eqref{firstineq} to considerably simpler one-dimensional Hardy type inequalities in the corresponding representation norms.

\begin{thm}\label{reduction} {\bf [First-order reduction principle]}
Let $\Omega$ be any  open set in $\rn$, $n \geq 2$. Assume that
$\mu$ is a measure in $\Omega$ fulfilling \eqref{measure} for some
$\alpha \in (n-1, n]$, and for some constant $C_\mu$. Let $X(\Omega)$, $Y(\Omega , \mu)$ and
 $Z(\partial \Omega)$  be rearrangement-invariant spaces such that
\begin{multline}\label{red1bis}
\Bigg\| \chi_{(0, \mu(\Omega))}(t)\Bigg(t^{-\frac {n-1}{\alpha}}\int _0^{t^{\frac
n\alpha}}\chi_{(0, \mathcal L^n(\Omega))}(\rho)f(\rho) d\rho\\ 
+\int _{t^{\frac n\alpha}}^\infty \rho^{-\frac{n-1}n}
\chi_{(0, \mathcal L^n(\Omega))}(\rho)f(\rho) d\rho \Bigg)\Bigg\|_{\overline Y(0, \infty )}   \leq C_1
\|\chi_{(0, \mathcal L^n(\Omega))}f\|_{\overline X(0, \infty )},
\end{multline}
and
\begin{equation}\label{red5}
\Bigg\|\chi_{(0, \mu(\Omega))}(t) t^{-\frac {n-1}{\alpha}}\int _0^{t^{\frac {n-1}\alpha}}
\chi_{(0, \hh(\partial \Omega))}(\rho)f(\rho)\, d\rho \Bigg\|_{\overline Y(0, \infty )} \leq C_2
\|\chi_{(0,  \hh(\partial \Omega))}f\|_{\overline Z(0, \infty )},
\end{equation}
for some constants $C_1$ and $C_2$, and for every non-increasing function
$f : [0, \infty) \to [0, \infty)$.   Then
\begin{align}\label{red6}
\|u\|_{Y(\Omega, \mu)} & \leq C' \Big(C_1 \|\nabla 
u\|_{X(\Omega )} +    C_2\|u_{\partial \Omega}\|_{ Z(\partial \Omega )}\Big)
\end{align}
  for some
constant $C'=C'(n)$ and for every $u \in {V^{1}X(\Omega )
\cap C_{\rm b}(\overline \Omega )}$.
\end{thm}

\begin{rem} One can verify that the expression appearing in each of the norms on the left-hand sides of inequalities  \eqref{red1bis} and \eqref{red5} is a nonnegative non-increasing function of $t$. Hence, it agrees with its decreasing rearrangement. This observation can be of use when dealing with rearrangement-invariant spaces $Y(\Omega, \mu)$, such as   Lorentz and Lorentz-Zygmund spaces, whose norms are defined in terms of rearrangements.
\end{rem}

\section{Second-order inequalities}

The second-order homogeneous Sobolev space built upon 
  a rearrangement-invariant space $X(\Omega)$  is defined as
\begin{equation}\label{sobolevV2}
V^{2}X (\Omega ) =   \big\{u: \hbox{$u$ is   twice-weakly
differentiable in $\Omega$, and $|\nabla^2 u| \in
X(\Omega)$}\big\}.
\end{equation}
Here, $\nabla^2 u$ denotes the matrix of all second-order derivatives of $u$.

As mentioned in Section \ref{S:introd}, upper gradients of the restriction to $\partial \Omega$ of trial functions come into play in our second-order inequalities.  
  Let  $\varphi  : \partial \Omega \to \mathbb R$ be a measurable function. 
A Borel function 
 $g_\varphi : \partial \Omega \to \mathbb R$ is called an upper gradient for  $\varphi$    in the sense of Hajlasz \cite{Hajlasz} if
\begin{equation}\label{hajlasz}
|\varphi(x) - \varphi(y)| \leq |x -y| (g_\varphi (x) + g_\varphi  (y)) \quad \hbox{for
$\hh$-a.e. $x, y \in
\partial \Omega$.}
\end{equation}
Given a rearrangement-invariant space $X(\partial \Omega)$ on $\partial \Omega$,   we denote by $\mathcal V^{1}X(\partial \Omega)$ the space of functions $\varphi$ that admit an upper gradient in $X(\partial \Omega)$.  It is easily verified that $\mathcal V^{1}X(\partial \Omega)$ is a linear space.  Furthermore, the functional defined as
\begin{equation}\label{semi}
\|\varphi\|_{\mathcal V^{1}X(\partial \Omega)} = \inf \|g_\varphi\|_{X(\partial \Omega)},
\end{equation}
where the infimum is extended over all upper gradients $g_\varphi$ of $\varphi$ in $X(\partial \Omega)$, is a seminorm on the space $\mathcal V^{1}X(\partial \Omega)$.

Our  estimates for $u$ involving Lebesgue norms of $\nabla ^2u$ are the content of the following theorem.

\begin{thm}\label{fried2_p<n/2} {\bf [Second-order   inequalities for $u$]}
Let $\Omega$ be any open  set in $\rn$, $n \geq 3$,  with $\mathcal L^n(\Omega)<\infty$ and $\hh (\Omega)< \infty$.   Assume that
$\mu$ is a   Borel measure in $\Omega$ fulfilling \eqref{measure} for
some $\alpha \in (n-1, n]$ and for some $C_\mu
>0$, and such that $\mu (\Omega)< \infty$. 
\\ (i)
Let $1<p<\frac n2$, $1 < s< n-1$,  $r>1$, and let 
\begin{equation}\label{q2}
q= \min \bigg\{\frac {p\alpha }{n-2p}, \frac{s\alpha}{n-1-s}, \frac {r\alpha }{n-1}\bigg\}.
\end{equation}
Then there exists a constant 
$C=C(n, p, r, s, \alpha, C_\mu)$ such that
\begin{align}\label{fried4}
\|u\|_{L^q(\Omega, \mu)} \leq C \Big(&\max\{\mu(\Omega)^{\frac n\alpha}, \mathcal L^n(\Omega)\}  ^{\frac \alpha{qn}-\frac{n-2p}{pn}} \|\nabla^2 u\|_{L^p(\Omega)}  \nonumber
\\ \nonumber &
+ 
 \max\{\mu(\Omega)^{\frac {n-1}\alpha}, \hh(\partial \Omega)\} ^{\frac \alpha{q(n-1)}-\frac{n-1-s}{s(n-1)}} \|u\|_{\mathcal V^1L^s(\partial \Omega)}
\\  r &
+
 \max\{\mu(\Omega)^{\frac {n-1}\alpha}, \hh(\partial \Omega)\} 
^{\frac \alpha{q(n-1)}-\frac 1{r}} \|u\|_{L^r(\partial \Omega)}\Big)
\end{align}
for every $u \in V^{2}L^p(\Omega ) \cap \mathcal V^{1}L^{s}(\partial \Omega ) \cap C_b(\overline \Omega )$. 
\\ (ii)
 Let $\beta>0$, $s>n-1$,  and let
$$\gamma= \min\{\tfrac n{n-2}, \beta \}.$$
Then there exists a constant 
$C=C(n,  	\beta, \alpha, C_\mu, \mathcal L^n(\Omega), \hh (\Omega), \mu(\Omega))$ such that
\begin{align}\label{fried5}
\|u\|_{\exp L^\gamma (\Omega, \mu)} \leq C \Big( \|\nabla^2 u\|_{L^{\frac n2}(\Omega)} + \|u\|_{\mathcal V^1L^s(\partial \Omega)}+ \|u\|_{\exp L^\beta(\partial \Omega)}\Big)
\end{align}
for every $u \in V^{2}L^n(\Omega ) \cap \mathcal V^{1}L^s(\partial \Omega) \cap C_b(\overline \Omega )$. 
\\ (iii)
Let
$p>\frac n2$ and $s> n-1$. Then there exists a constant $C=C(n, p,  s)$ such that
\begin{align}\label{fried7}
\| u\|_{L^{\infty}(\Omega)}  \leq C \big(\mathcal L^n (\Omega)^{\frac 2n - \frac 1{p}} \|\nabla ^{2}
u\|_{L^p(\Omega )} + 
\hh (\partial \Omega)^{\frac 1{n-1}- \frac 1s} \|u\|_{\mathcal V^{1}L^{s}(\partial \Omega)} + \|u\|_{L^{\infty}(\partial \Omega )}\big)
\end{align}
for every $u \in V^{2}L^p(\Omega ) \cap \mathcal V^{1}L^{s}(\partial \Omega) \cap C_b(\overline \Omega )$
\end{thm}

\begin{rem} In the  doubly borderline case when $p=\frac n2$ and $s=n-1$, Part (ii) of Theorem \ref{fried2_p<n/2}  adimts the following variant. Let $\beta>0$ and 
$$\gamma= \min\big\{\tfrac {n-1}{n-2}, \beta \big\}.$$
Then there exists a constant 
$C=C(n,	\beta, \alpha, C_\mu, \mathcal L^n(\Omega), \hh (\Omega), \mu(\Omega))$ such that
\begin{align}\label{fried6}
\|u\|_{\exp L^\gamma (\Omega, \mu)} \leq C \Big( \|\nabla^2 u\|_{L^{\frac n2}(\Omega)} + \|u\|_{\mathcal V^1L^{n-1}(\partial \Omega)}+ \|u\|_{\exp L^\beta(\partial \Omega)}\Big)
\end{align}
for every $u \in V^{2}L^n(\Omega ) \cap \mathcal V^{1}L^{n-1}(\partial \Omega) \cap C_b(\overline \Omega )$. 
\end{rem}

The compactness of embeddings associated with the norms defined via the right-hand sides of the inequalities of Theorem \ref{fried2_p<n/2} is discussed in the next result.

\begin{thm}\label{compact2_sub}{\bf [Second-order compact   embeddings]}
Let $\Omega$ be any open  set in $\rn$, $n \geq 3$ with $\mathcal L^n(\Omega)<\infty$ and $\hh (\Omega)< \infty$.  Let
$\mu$ be a   Borel measure in $\Omega$ fulfilling \eqref{measure} for
some $\alpha \in (n-1, n]$ and for some $C_\mu
>0$, and such that $ \mu (\Omega)< \infty$.  
\\ (i) Let $1<p<\frac n2$,  $1 < s< n-1$ and  $r>1$.
Assume that
\begin{equation}\label{compnew3}
1 \leq q <\min \bigg\{\frac {p\alpha }{n-2p}, \frac{s\alpha}{n-1-s}, \frac {r\alpha }{n-1}\bigg\}.
\end{equation}
If
 $\{u_k\}$
is a bounded sequence in the space $V^{2}L^p(\Omega ) \cap \mathcal V^{1}L^s(\partial \Omega) \cap
 C_{\rm b}(\overline \Omega )$ endowed with the norm appearing
on the right-hand side of \eqref{fried4},  then $\{
u_k\}$ is a Cauchy sequence in $L^q(\Omega, \mu)$.
\\ (ii)
Let $p=\frac n2$ and $s>n-1$. Assume that
\begin{equation}\label{compnew4}
0<\gamma < \min\{\tfrac n{n-2}, \beta \}.
\end{equation}
 If
 $\{u_k\}$
is a bounded sequence in the space
$u \in V^{2}L^n(\Omega ) \cap \mathcal V^{1}L^s(\partial \Omega) \cap C_b(\overline \Omega )$ endowed with the norm appearing
on the right-hand side of inequality \eqref{fried5},  then $\{
u_k\}$ is a Cauchy sequence in $\exp L^\gamma (\Omega, \mu)$.
\end{thm}

Theorem \ref{fried2_p<n/2} has a counterpart, dealing with bounds for $\nabla u$ instead of $u$, which reads as follows.

\begin{thm}\label{fried2_grad_p<n} {\bf [Second-order   inequality for $\nabla u$]}
Let $\Omega$ be any open  set in $\rn$, $n \geq 2$, with $\mathcal L^n(\Omega)<\infty$ and $\hh (\Omega)< \infty$.  Assume that
$\mu$ is a   Borel measure in $\Omega$ fulfilling \eqref{measure} for
some $\alpha \in (n-1, n]$ and for some $C_\mu
>0$, and such that $\mu (\Omega)< \infty$. 
\\ (i) 
Let $1<p<n$, $r>1$, and let  $q$ be defined as in \eqref{q}. 
Then there exists a constant 
$C=C(n,  p, r, \alpha, C_\mu)$ such that
\begin{align}\label{fried8}
\|\nabla u\|_{L^q(\Omega, \mu)}& \leq C \Big(\max\{\mu(\Omega)^{\frac n\alpha}, \mathcal L^n(\Omega)\}^{\frac \alpha{qn}-\frac{n-p}{pn}} \|\nabla^2 u\|_{L^p(\Omega)} \nonumber
\\   &   \qquad + 
\max\{\mu(\Omega)^{\frac n\alpha}, \hh(\partial \Omega)\}^{\frac \alpha{q (n-1)}-\frac{1}{r}} \|u\|_{\mathcal V^1 L^r(\partial \Omega)}\Big)
\end{align}
for every $u \in V^{2}L^p(\Omega ) \cap \mathcal V^{1}L^r(\partial \Omega) \cap C_b(\overline \Omega )$. 
\\ (ii) 
Let $\beta>0$,  and let $\gamma$ be defined as in \eqref{gamma}.
Then there exists a constant 
$C=C(n,	\beta, \alpha, C_\mu, \mathcal L^n (\Omega), \hh (\Omega), \mu(\Omega))$ such that
\begin{align}\label{fried9}
\|\nabla u\|_{\exp L^\gamma (\Omega, \mu)} \leq C \Big( \|\nabla^2 u\|_{L^n(\Omega)} +  \|u\|_{\mathcal V^1\exp L^\beta(\partial \Omega)}\Big)
\end{align}
for every $u \in V^{2}L^n(\Omega ) \cap \mathcal V^{1}\exp L^\beta(\partial \Omega) \cap C_b(\overline \Omega )$. 
\\(iii) Let $p
>n$.  Then
 there
 exists a  constant $C=C(n, p)$
  such that
\begin{align}\label{inf2}
\|\nabla u\|_{L^\infty (\Omega )} & \leq C \Big(\mathcal L^n(\Omega)^{\frac 1n-\frac 1p}\|\nabla^2 u\|_{L^{p}(\Omega )}
 +     \|u\|_{\mathcal V^1L^\infty(\partial \Omega)}\Big)
\end{align}
for every $u \in V^{2}L^p(\Omega) \cap \mathcal V^{1}L^\infty (\partial \Omega)\cap C_b(\overline \Omega)$.
\end{thm}

Inequalities for functions in second-order Lorentz-Sobolev spaces, with improved target spaces, can also be derived. The conclusions about bounds for $u$ are collected in Theorem \ref{Lorentz2}. An analogue concerning estimates for $\nabla u$ holds, and calls into play the same norms as in Theorem \ref{Lorentz}. Its statement is omitted for brevity.

\begin{thm}\label{Lorentz2} {\bf [Second-order Lorentz--Sobolev inequalities]}
Let $\Omega$ be any open  set in $\rn$, $n \geq 3$,  with $\mathcal L^n(\Omega)<\infty$ and $\hh (\Omega)< \infty$.   Assume that
$\mu$ is a   Borel measure in $\Omega$ fulfilling \eqref{measure} for
some $\alpha \in (n-1, n]$ and for some $C_\mu
>0$, and such that $\mu (\Omega)< \infty$. 
\\ (i) Let $1<p<\frac n2$, $1 < s< n-1$,  $r>1$,  $1 \leq \sigma, \upsilon, \varrho \leq \infty$. 
Let  $q$ be defined as in \eqref{q2}  
and let
$$\eta \geq 
\begin{cases}
\sigma & \quad \text{if $q=\frac {p\alpha }{n-2p}$}
\\
\upsilon & \quad \text{if $q= \frac{s\alpha}{n-1-s}$}
\\ 
\varrho  & \quad \text{if $q=\frac {r\alpha }{n-1}$.}
\end{cases}
$$
Then there exists a constant 
$C=C(n, p, s, r, \eta, \sigma, \upsilon, \alpha, \varrho, C_\mu)$ such that
\begin{align}\label{L2}
\|u\|_{L^{q,\eta}(\Omega, \mu)} \leq C \Big(&\max\{\mu(\Omega)^{\frac n\alpha}, \mathcal L^n(\Omega)\}  ^{\frac \alpha{qn}-\frac{n-2p}{pn}} \|\nabla^2 u\|_{L^{p,\sigma}(\Omega)}  \nonumber
\\ \nonumber &
+ 
 \max\{\mu(\Omega)^{\frac {n-1}\alpha}, \hh(\partial \Omega)\} ^{\frac \alpha{q(n-1)}-\frac{n-1-s}{s(n-1)}} \|u\|_{\mathcal V^1 L^{s, \upsilon}(\partial \Omega)}
\\  &
+
 \max\{\mu(\Omega)^{\frac {n-1}\alpha}, \hh(\partial \Omega)\} 
^{\frac \alpha{q(n-1)}-\frac 1{r}} \|u\|_{L^{r, \varrho}(\partial \Omega)}\Big)
\end{align}
for every $u \in V^{2}L^{p, \eta}(\Omega ) \cap \mathcal V^{1}L^{s,\upsilon}(\partial \Omega ) \cap C_b(\overline \Omega )$. 
\\ (ii) Assume that   $\sigma , \upsilon,  \varrho >1$ and $\varsigma < - \tfrac 1\varrho$.  Let 
$$\eta \geq \max \{\sigma, \upsilon,  \varrho\},$$
and 
$$\theta \leq  \min \Big\{-1+ \tfrac 1\sigma - \tfrac 1\eta, - 1 + \tfrac 1\upsilon-\tfrac 1 \eta, \varsigma + \tfrac 1 \varrho -\tfrac 1\eta\Big\}.$$
Then there
 exists a constant $C=C(n, \eta, \sigma, \varrho, \theta, \upsilon, \varsigma, \alpha, C_\mu, \mathcal L^n (\Omega), \mu(\Omega), \hh (\partial \Omega))$ such that
 \begin{align}\label{BW2}
\|u\|_{L^{\infty, \eta; \theta}(\Omega, \mu)}  \leq
C\Big( \|\nabla^2 u\|_{L^{\frac n2,\sigma}(\Omega )}
  + \|u\|_{\mathcal V^1L^{n-1, \upsilon}(\partial \Omega)}+
 \|u\|_{L^{\infty, \varrho; \varsigma}(\partial \Omega)}\Big)
\end{align}
for every $u \in V^{1}L^{n,\sigma}(\Omega) \cap \mathcal V ^{1} L^{n-1, \upsilon}(\partial \Omega) \cap C_b(\overline \Omega )$. 
\\ (iii) Assume that either $p=\frac n2$ and $\sigma =1$, or $p>n$ and $1\leq \sigma\leq \infty$, and either $s=n-1$ and $\upsilon =1$, or $s>n-1$ and $1 \leq \upsilon \leq \infty$. 
Then there
 exists a constant $C=C(n, p, \sigma, s, \upsilon, \alpha, C_\mu)$ such that
\begin{align}\label{infL2}
\|u\|_{L^\infty (\Omega )} & \leq C \Big(\mathcal L^n(\Omega)^{\frac 2n-\frac 1p}\|\nabla^2 u\|_{L^{p,\sigma}(\Omega )} + \hh (\partial \Omega)^{\frac 1{n-1}- \frac 1s}  \|u\|_{\mathcal V^1 L^{s, \upsilon}(\partial \Omega)}
 +     \|u\|_{L^\infty(\partial \Omega)}\Big)
\end{align}
for every $u \in V^{1}L^{p,\sigma}(\Omega) \cap \mathcal V^1L^{s, \upsilon}(\partial \Omega) \cap C_b(\overline \Omega)$.
\end{thm}

\subsection{Second-order pointwise estimates and reduction principle}\label{2est}

We collect here second-order versions of the results of Section \ref{1est}, namely pointwise estimates, rearrangement-estimates, and reduction principles, for both   $u$ and   $\nabla u$.

\begin{thm}\label{point2}
{\bf [Second-order pointwise estimates]} Let $\Omega$ be any  open set in $\rn$.
\\ (i) Assume that $n \geq 3$. There exists a constant $C=C(n)$ such that
\begin{align}\label{point2.1}
|u(x)| & \leq C \bigg(\int _\Omega \frac{|\nabla
^{2}u(y)|}{|x-y|^{n-2}}\, dy + \int
_{\Omega}  \int _{\mathbb S^{n-1}}\frac{g_u(\zeta (y,\vartheta ))}{|x-y|^{n-1}} \, d\hh (\vartheta ) \,
dy \nonumber
\\  & \qquad \qquad +    \int _{\mathbb S^{n-1}}|u_{\partial \Omega}(\zeta (x, \vartheta ))|\,
d\hh (\vartheta )\bigg) \qquad \hbox{for   $x \in \Omega$,}
\end{align}
for every  $u \in V^{2}L^1(\Omega ) \cap \mathcal V^{1}L^1(\partial \Omega) \cap C_{\rm
b}(\overline \Omega )$. 
\\ (ii)  Assume that $n \geq 2$. There exists a constant $C=C(n)$ such that
\begin{align}\label{point2.2}
|\nabla u(x)|  \leq C \Bigg(\int _\Omega \frac{|\nabla
^{2}u(y)|}{|x-y|^{n-1}}\, dy + \int _{\mathbb S^{n-1}}g_u(\zeta (x, \vartheta ))\,
d\hh (\vartheta )\Bigg)
 \qquad \hbox{for a.e. $x \in \Omega$,}
\end{align}
for every  $u \in V^{2}L^1(\Omega ) \cap \mathcal V^{1}L^1(\partial \Omega) \cap C_{\rm
b}(\overline \Omega )$. 
\\ In inequalities \eqref{point2.1} and \eqref{point2.2}, $g_u$ denotes any Hajlasz gradient of $u_{\partial \Omega}$, and convention
\eqref{conv} is adopted.
\end{thm}

\begin{thm}\label{rearrest_2} {\bf [Second-order rearrangement estimates]}
Let $\Omega$ be any open  set in $\rn$.   Assume that
$\mu$ is a Borel measure in $\Omega$ fulfilling \eqref{measure} for
some $\alpha \in (n-1, n]$ and for some $C_\mu
>0$. 
\\ (i) Assume that $n \geq 3$. There exist  constants $c=c(n)$ and $C=C(n,\alpha ,
C_\mu)$ such that
\begin{align}\label{rearr2u}
u_\mu^*(ct)  & \leq C\Bigg(t^{-\frac
{n-2}{\alpha}}\int _0^{t^{\frac n\alpha}}|\nabla ^{2} u|_{\mathcal
L^n}^*(\rho) d\rho + \int _{t^{\frac n\alpha}}^\infty r^{-\frac{n-2}n}
|\nabla ^{2} u|_{\mathcal L^n}^*(\rho) d\rho \nonumber
\\ \nonumber & \quad \quad \quad +
 t^{-\frac {n-2}{\alpha}}\int
_0^{t^{\frac {n-1}\alpha}}  \big(g _u \big)_{\hh}^*(\rho)d\rho
  +  \int_{t^{\frac
{n-1}\alpha}}^\infty \rho^{-\frac {n-2}{n-1}}
\big(g_u \big)_{\hh}^*(\rho)d\rho
\\   & \quad \quad \quad + t^{-\frac {n-1}{\alpha}}\int _0^{t^{\frac {n-1}\alpha}}
\big(u_{\partial \Omega}\big)_{\hh}^*(\rho)d\rho\Bigg)\quad
\quad \hbox{for $t>0$,}
\end{align}
for every $u \in V^{2}L^1(\Omega ) \cap \mathcal V^{1}L^1(\partial \Omega) \cap C_{\rm
b}(\overline \Omega )$.  
\\ (ii) Assume that $n \geq 2$. There exist  constants $c=c(n)$ and $C=C(n,\alpha ,
C_\mu)$ such that
\begin{align}\label{rearr2gradu}
|\nabla u|_\mu^*(ct)  & \leq C\Bigg(t^{-\frac
{n-1}{\alpha}}\int _0^{t^{\frac n\alpha}}|\nabla^2 u|_{\mathcal
L^n}^*(\rho) d\rho+ \int _{t^{\frac n\alpha}}^\infty \rho^{-\frac{n-1}n}
|\nabla^2 u|_{\mathcal L^n}^*(\rho) d\rho \nonumber
\\   & \quad \quad \quad + t^{-\frac {n-1}{\alpha}}\int _0^{t^{\frac {n-1}\alpha}}
\big(g_u\big)_{\hh}^*(\rho)d\rho\Bigg)\quad
\quad \hbox{for $t>0$,}
\end{align}
for every $u \in V^{2}L^1(\Omega ) \cap \mathcal V^{1}L^1(\partial \Omega)  \cap C_{\rm
b}(\overline \Omega )$. 
\\ In inequalities \eqref{rearr2u}  and \eqref{rearr2gradu}, $g_u$ denotes
any Hajlasz gradient of $u_{\partial \Omega}$.
\end{thm}

\begin{thm}\label{reduction2} {\bf [Second-order reduction principles]}
Let $\Omega$ be any  open set in $\rn$. Assume that
$\mu$ is a measure in $\Omega$ fulfilling \eqref{measure} for some
$\alpha \in (n-1, n]$, and for some constant $C_\mu$.  
\\ (i)  Assume that $n \geq 3$. Let $X(\Omega)$, $Y(\Omega , \mu)$, 
 $U(\partial \Omega)$ and $Z(\partial \Omega)$  be rearrangement-invariant spaces such that
\begin{multline}\label{red2.1bis}
\Bigg\| \chi_{(0, \mu(\Omega))}(t)  \Bigg(t^{-\frac {n-2}{\alpha}}\int _0^{t^{\frac
n\alpha}}\chi_{(0, \mathcal L^n(\Omega))}(\rho)f(\rho) d\rho \\ + \int _{t^{\frac n\alpha}}^\infty \rho^{-\frac{n-2}n}\chi_{(0, \mathcal L^n(\Omega))}(\rho)
f(\rho) d\rho\Bigg) \Bigg\|_{\overline Y(0, \infty )} \leq C_1
\|\chi_{(0, \mathcal L^n(\Omega))}f\|_{\overline X(0, \infty )},
\end{multline}
\begin{multline}\label{red2.3bis}
\Bigg\|\chi_{(0, \mu(\Omega))}(t)\Bigg(t^{-\frac {n-2}{\alpha}}\int _0^{t^{\frac {n-1}\alpha}}\chi_{(0, \hh(\partial \Omega))}(\rho)
f(\rho)d\rho \\ + \int_{t^{\frac {n-1}\alpha}}^\infty  \rho^{-\frac {n-2}{n-1}}\chi_{(0, \hh(\partial \Omega))}(\rho)
f(\rho)d\rho \Bigg)\Bigg\|_{\overline Y(0, \infty )} \leq C_2
\|\chi_{(0, \hh(\partial \Omega))}f\|_{\overline U(0, \infty )} 
\end{multline}
\begin{multline}\label{red2.5}
\Bigg\| \chi_{(0, \mu(\Omega))}(t) t^{-\frac {n-1}{\alpha}}\int _0^{t^{\frac {n-1}\alpha}}\chi_{(0, \hh(\partial \Omega))}(\rho)
f(\rho)\, d\rho \Bigg\|_{\overline Y(0, \infty )}\\  \leq C_3
\|\chi_{(0, \hh(\partial \Omega))}f\|_{\overline Z(0, \infty )},
\end{multline}
for some constants $C_1$, $C_2$ and $C_3$, and for every non-increasing function
$f : [0, \infty) \to [0, \infty)$. 
Then
\begin{align}\label{red2.6}
\| u\|_{Y(\Omega, \mu)} & \leq C' \Big(C_1\|\nabla ^{2}
u\|_{X(\Omega )} +   C_2\|u\|_{\mathcal V^1
U(\partial \Omega )} + C_3\|u_{\partial \Omega}\|_{
Z(\partial \Omega )}  \Big)
\end{align}
for some
constant $C'=C'(n)$ and
 for every $u \in {V^{2}X(\Omega )
\cap \mathcal V^1U(\partial \Omega) \cap C_{\rm b}(\overline \Omega )}$.
\\ (ii)   Assume that $n \geq 2$. Let $X(\Omega)$, $Y(\Omega , \mu)$ 
  and $Z(\partial \Omega)$  be rearrangement-invariant spaces such that inequalities \eqref{red1bis} and \eqref{red5} hold. Then 
\begin{align}\label{red6second}
\|\nabla u\|_{Y(\Omega, \mu)} & \leq C' \Big(C_1 \|\nabla^2
u\|_{X(\Omega )} +    C_2\|u \|_{\mathcal V^1 Z(\partial \Omega )}\Big)
\end{align}
  for some
constant $C'=C'(n)$ and for every $u \in {V^{2}X(\Omega )\cap \mathcal V^1Z(\partial \Omega)
\cap C_{\rm b}(\overline \Omega )}$.
\end{thm}

\begin{rem} The expression appearing in each of the norms on the left-hand sides of inequalities  \eqref{red2.1bis}--\eqref{red2.5} is a nonnegative non-increasing function of $s$. Therefore, it agrees with its decreasing rearrangement. 
\end{rem}

\section{Inequalities for the symmetric gradient}

In this section we deal with Friedrichs type inequalities, in the spirit of those presented in Section \ref{S:introd}, involving just the symmetric gradient  $\mathcal E {\bf u}$ of functions   $ {\bf u} : \Omega \to \rn$.  
\\
In analogy with \eqref{sobolevV}, given a rearrangement-invariant space 
$X(\Omega)$,  we define the symmetric gradient Sobolev space as
\begin{equation}\label{EX}E^1X(\Omega) = \big\{{\bf u} \in L^1_{\rm loc}(\Omega): \, |\mathcal E {\bf u}| \in X(\Omega)\big\}.
\end{equation}
Interestingly, it turns out that the norms in the relevant symmetric gradient  inequalities are exactly the same as those entering their full gradient counterparts. As an example, we reproduce here the basic conclusions concerning functions in the space $E^1L^p(\Omega)$, with $p\in (1, \infty]$. A related result   can also be found in the recent paper \cite{ChMaz}.

\begin{thm}\label{fried_symm} {\bf [Symmetric gradient   inequalities]}
Let $\Omega$ be any open  set in $\rn$, $n \geq 2$ with $\mathcal L^n(\Omega)<\infty$ and $\hh (\Omega)< \infty$.  Assume that
$\mu$ is a   Borel measure in $\Omega$ fulfilling \eqref{measure} for
some $\alpha \in (n-1, n]$ and for some $C_\mu
>0$, and such that $ \mu (\Omega)< \infty$. 
\\ (i) 
Let $1<p<n$,  $r>1$, and let  $q$ be defined as in \eqref{q}.
Then there exists a constant 
$C=C(n,p, r, \alpha,  C_\mu)$ such that
\begin{align}\label{friedsymm.1}
\|{\bf u}\|_{L^q(\Omega, \mu)}& \leq C \Big(\max\{\mu(\Omega)^{\frac n\alpha}, \mathcal L^n(\Omega)\}   ^{\frac \alpha{qn}-\frac{n-p}{pn}} \|\mathcal E {\bf u}\|_{L^p(\Omega)}  \nonumber
\\ & \qquad + \max\{\mu(\Omega)^{\frac {n-1}\alpha}, \hh(\partial \Omega)\}  ^{\frac \alpha{q(n-1)}-\frac{1}{r}} \|{\bf u}\|_{L^r(\partial \Omega)}\Big)
\end{align}
for every ${\bf u} \in E^{1}L^p(\Omega ) \cap C_b(\overline \Omega )$. 
\\ (ii)
Let $\beta>0$,  and let $\gamma$ be defined as in \eqref{gamma}.
Then there exists a constant 
$C=C(n, 	\beta, \alpha,  C_\mu, \mathcal L^n (\Omega), \hh (\Omega), \mu(\Omega))$ such that
\begin{align}\label{friedsymm.3}
\|{\bf u}\|_{\exp L^\gamma (\Omega, \mu)} \leq C \Big( \|\mathcal E {\bf u}\|_{L^n(\Omega)} +  \|{\bf u}\|_{\exp L^\beta(\partial \Omega)}\Big)
\end{align}
for every ${\bf u} \in E^{1}L^n(\Omega ) \cap C_b(\overline \Omega )$. 
\\ (iii)
 Let $p
>n$. Then
 there
 exists a  constant $C=C(n, p)$
  such that
\begin{align}\label{infsymm}
\|{\bf u}\|_{L^\infty (\Omega )} & \leq C \Big(\mathcal L^n(\Omega)^{\frac 1n-\frac 1p}\|\mathcal E {\bf u}\|_{L^{p}(\Omega )}
 +     \|{\bf u}\|_{L^\infty(\partial \Omega)}\Big)
\end{align}
for every ${\bf u} \in V^{1}L^p(\Omega) \cap C_b(\overline \Omega)$.
\end{thm}

\subsection{Symmetric-gradient pointwise estimates and reduction principle}\label{symmest}

 Our approach to Friedrichs inequalities in the spaces $E^1X(\Omega)$ is grounded on
pointwise and rearrangement bounds in terms of the symmetric gradient. They  parallel those presented in Section  \ref{1est} for the full gradient and are exposed in the next two theorems.

\begin{thm}\label{pointsymm}
{\bf [Symmetric gradient pointwise estimate]} Let $\Omega$ be any  open set in $\rn$,
$n \geq 2$.  Then there exists a constant $C=C(n)$ such that
\begin{align}\label{point1sym}
|{\bf u}(x)| & \leq C \Bigg(\int _\Omega \frac{|\mathcal E {\bf u}(y)|}{|x-y|^{n-1}}\, dy +    \int _{\mathbb S^{n-1}}|{\bf u}_{\partial \Omega}(\zeta (x, \vartheta ))|\,
d\hh (\vartheta )\Bigg) \qquad \hbox{for  $x \in \Omega$,}
\end{align}
for every  ${\bf u} \in E^{1 }L^1(\Omega ) \cap C_{\rm
b}(\overline \Omega )$. Here,   convention
\eqref{conv} is adopted.
\end{thm}

\begin{thm}\label{rearrest_1symm} {\bf [Symmetric gradient rearrangement estimate]}
Let $\Omega$ be any open  set in $\rn$, $n \geq 2$.   Assume that
$\mu$ is a Borel measure in $\Omega$ fulfilling \eqref{measure} for
some $\alpha \in (n-1, n]$ and for some $C_\mu
>0$. Then there exist  constants $c=c(n)$ and $C=C(n,\alpha ,
C_\mu)$ such that
\begin{align}\label{rearrsym}
|{\bf u}|_\mu^*(ct)  & \leq C\Bigg(t^{-\frac
{n-1}{\alpha}}\int _0^{t^{\frac n\alpha}}|\mathcal E {\bf u}|_{\mathcal
L^n}^*(\rho) d\rho + \int _{t^{\frac n\alpha}}^\infty \rho^{-\frac{n-1}n}
|\mathcal E {\bf u}|_{\mathcal L^n}^*(\rho) d\rho  \nonumber \\  
& \quad \quad \quad + t^{-\frac {n-1}{\alpha}}\int _0^{t^{\frac {n-1}\alpha}}
|{\bf u}_{\partial \Omega}|_{\hh}^*(\rho)d\rho\Bigg)\quad
\quad \hbox{for $t>0$,}
\end{align}
for every ${\bf u} \in E^{1}L^1(\Omega ) \cap C_{\rm
b}(\overline \Omega )$.  
\end{thm}

As a consequence of Theorem \ref{rearrest_1symm}, a reduction principle for symmetric gradient inequalities takes exactly the same form as that stated in Theorem \ref{reduction}.

\begin{thm}\label{reduction_symm} {\bf [Symmetric gradient reduction principle]}
Let $\Omega$ be any  open set in $\rn$, $n \geq 2$. Assume that
$\mu$ is a measure in $\Omega$ fulfilling \eqref{measure} for some
$\alpha \in (n-1, n]$, and for some constant $C_\mu$. Let $X(\Omega)$, $Y(\Omega , \mu)$ and
 $Z(\partial \Omega)$  be rearrangement-invariant spaces such that inequalities \eqref{red1bis}--\eqref{red5} hold for some constants $C_1$ and $C_2$.
 Then
\begin{align}\label{redsym}
\|{\bf u}\|_{Y(\Omega, \mu)} & \leq C' \Big(C_1\|\mathcal E 
{\bf u}\|_{X(\Omega )} +   C_2 \|{\bf u}_{\partial \Omega}\|_{ Z(\partial \Omega )}\Big)
\end{align}
for some
constant $C'=C'(n)$ and 
 for every ${\bf u} \in {E^{1}X(\Omega )
\cap C_{\rm b}(\overline \Omega )}$.
\end{thm}

\section{Sketches of proofs}\label{proofs}

We conclude by outlining the proofs of Theorems \ref{fried2_p<n/2} and \ref{compact2_sub}. This should  help the reader grasp  methods to derive    inequalities  and   compactness results    via our reduction principles. For  proofs of the latter we refer to the papers \cite{CianchiMazya1, CianchiMazya2}.

\begin{proof}
[Proof of Theorem \ref{fried2_p<n/2}] \emph{Part (i)} By Theorem \ref{reduction2} and a change of variables, inequality \eqref{fried4} will follow if we show that:
\begin{multline}\label{fried35}
\Bigg\| t^{-\frac {n-2}{n}+(\frac \alpha n-1)\frac 1q}\int _0^{t}f(\rho) d\rho \Bigg\|_{L^q(0, \ell_1)}   +  \Bigg\| t^{(\frac \alpha n-1)\frac 1q}\int _{t}^{\ell_1} \rho^{-\frac{n-2}n}
f(\rho) d\rho\Bigg\|_{L^q(0, \ell_1)} \\ \leq C_1
\|f\|_{ L^p(0,\ell_1)}
\end{multline}
for every non-increasing function $f: (0,\ell_1) \to [0, \infty)$, where $\ell_1= \max\{\mu(\Omega)^{\frac n\alpha}, \mathcal L^n(\Omega)\}$, $C_1 = c \ell_1^{\frac \alpha{qn}-\frac{n-2p}{pn}}$ and $c=c(n, p, r, s, \alpha, C_\mu)$;
\begin{multline}\label{fried37}
\Bigg\|t^{-\frac {n-2}{n-1}+(\frac \alpha {n-1}-1)\frac 1q}\int _0^{t}
f(\rho)d\rho \Bigg\|_{L^q(0, \ell_2 )}+ \Bigg\|t^{(\frac \alpha {n-1}-1)\frac 1q} \int_{t}^{\ell_2}  \rho^{-\frac {n-2}{n-1}}
f(\rho)d\rho \Bigg\|_{L^q(0, \ell_2)} \\ \leq C_2
\|f\|_{L^s(0, \ell_2)} 
\end{multline}
for every non-increasing function $f: (0,\ell_2) \to [0, \infty)$, where $\ell_2=  \max\{\mu(\Omega)^{\frac {n-1}\alpha}, \hh(\partial \Omega)\} $, $C_2 = c \ell_2^{\frac \alpha{q(n-1)}-\frac{n-1-s}{s(n-1)}}$ and $c=c(n,p, r, s, \alpha,  C_\mu)$;
\begin{equation}\label{fried39}
\Bigg\| t^{-1+(\frac \alpha {n-1}-1)\frac 1q}\int _0^{t}
f(\rho)\, d\rho \Bigg\|_{L^q(0, \ell_2)} \leq C_3
\|f\|_{L^r(0, \ell_2)}
\end{equation}
for every non-increasing function $f: (0,\ell_2) \to [0, \infty)$, where  $C_3 = c \ell_2^{\frac \alpha{q(n-1)}-\frac 1r}$ and $c=c(n,p, r, s, \alpha,  C_\mu)$.
\\ Inequalities  \eqref{fried35}--\eqref{fried39}  can be established
via standard criteria for one-dimensional Hardy type inequalities  -- see e.g. \cite[Section 1.3.2]{Mabook}. 
\\ \emph{Part (ii)} Owing to Theorem \ref{reduction2}, the subsequent remark,  equation \eqref{equivexp}, and a change of variables, the proof of inequality \eqref{fried5} is reduced to showing that: 
\begin{multline}\label{fried30}
\Bigg\| \Bigg(t^{-\frac {n-2}{n}}\int _0^{t}f(\rho) d\rho \Bigg)\log^{-\frac 1\gamma}\bigg(1+ \frac {\ell_1}{t^{\frac \alpha n}}\bigg) \Bigg\|_{L^\infty(0, \ell_1)} \\ + 
\Bigg\|\Bigg( \int _{t }^{\ell_1} \rho^{-\frac{n-2}n}
f(\rho) d\rho\Bigg) \log^{-\frac 1\gamma}\bigg(1+ \frac {\ell_1}{t^{\frac \alpha n}}\bigg)\Bigg\|_{L^\infty(0, \ell_1)}
\leq C_1
\|f\|_{ L^{\frac n2}(0,\ell_1)}
\end{multline}
for every non-increasing function $f: (0,\ell_1) \to [0, \infty)$ and for some constant 
$C_1=C_1(n,\alpha,   C_\mu, \mu(\Omega), \mathcal L^n(\Omega) )$;
\begin{multline}\label{fried32}
\Bigg\|\Bigg(t^{-\frac {n-2}{n-1}}\int _0^{t}
f(\rho)d\rho \Bigg) \log^{-\frac 1\gamma}\bigg(1+ \frac {\ell_2}{t^{\frac \alpha {n-1}}}\bigg)\Bigg\|_{L^\infty(0, \ell_2 )}\\
+ \Bigg\| \Bigg(\int_{t}^{\ell_2}  r^{-\frac {n-2}{n-1}}
f(r)dr \Bigg) \log^{-\frac 1\gamma}\bigg(1+ \frac {\ell_2}{t^{\frac \alpha {n-1}}}\bigg)\Bigg\|_{L^\infty(0, \ell_2)}
 \leq C_2
\|f\|_{L^s(0, \ell_2)} 
\end{multline}
for every non-increasing function $f: (0,\ell_2) \to [0, \infty)$   and for some constant 
$C_2=C_2(n,\alpha, s, C_\mu, \mu(\Omega), \hh(\partial \Omega) )$;
\begin{equation}\label{fried34}
\Bigg\| \Bigg(t^{-1}\int _0^{t }
f(\rho)\, d\rho \Bigg)\log^{-\frac 1\gamma}\bigg(1+ \frac {\ell_2}{t^{\frac \alpha {n-1}}}\bigg)\Bigg\|_{L^\infty(0, \ell_2)} \leq C_3
\bigg\|f (t)\log^{-\frac 1\gamma}\bigg(1+ \frac {\ell_2}{t^{\frac \alpha {n-1}}}\bigg)\bigg\|_{L^\infty(0, \ell_2 )}
\end{equation}
for every non-increasing function $f: (0,\ell_2) \to [0, \infty)$  and for some constant 
$C_3=C_3(n, r, \alpha,  C_\mu, \mu(\Omega), \hh(\partial \Omega) )$.
\\ Inequalities  \eqref{fried30}--\eqref{fried34}  follow via the criteria for weighted one-dimensional Hardy type inequalities  mentioned above. 
\\ \emph{Part (iii)} The proof of inequality \eqref{fried7} is analogous to that of inequality \eqref{fried5}. One has just  to set $\mu=\mathcal L^n$, $\alpha =n$, $q=\infty$ in inequalities \eqref{fried35}--\eqref{fried39}, and derive the correct dependence of the constants $C_1$, $C_2$ and $C_3$ from appropriate results for   Hardy type inequalities  \cite[Section 1.3.2]{Mabook}. 
\end{proof}

\begin{proof}
[Proof of Theorem \ref{compact2_sub}] Part (i). Fix any $\varepsilon >0$.
Then, there exists a compact set $K\subset \Omega$ such that $\mu
(\Omega \setminus K)< \varepsilon$.  Let $\xi \in C_0^\infty
(\Omega)$ be such that $0 \leq \xi \leq 1$, $\xi = 1$ in
$K$. Thus, $K \subset  {\rm supp} (\xi)$, the support of
$\xi$, and hence $\mu ({\rm supp} (1-\xi )) \leq \mu(\Omega \setminus K)<
\varepsilon$.
 Let $\Omega '$ be an open set, with a smooth boundary,
such that ${\rm supp}(\xi) \subset \Omega ' \subset \Omega$.
 Let $\{u_k\}$ be a bounded sequence in the space 
${V^{2}L^p(\Omega ) \cap \mathcal V^{1}L^s(\partial \Omega) \cap C _{\rm b}(\overline
\Omega )}$.
Thus, owing to an application of Theorem \ref{fried2_p<n/2},  with $\mu = \mathcal
L^n$, such a sequence  is also bounded in the standard Sobolev space
$W^{{2},p}(\Omega ')$. A weighted version of Rellich's
compactness theorem \cite[Theorem 1.4.6/1]{Mabook}, ensures that 
$\{ u_k\}$ is a Cauchy sequence in $L^{q}(\Omega ', \mu)$. As a consequence,
 there exists $k_0 \in \mathbb N$ such that
\begin{equation}\label{comp1}
\| u_k -  u_j\|_{L^q(\Omega ', \mu)}<\varepsilon
\end{equation}
if $k, j > k_0$. Now, denote by $\widehat q$ the minimum in equation \eqref{compnew3}.  By H\"older's inequality,
\begin{align}\label{comp2}
\|(1-\xi ) (u_k - u_j)\|_{L^q(\Omega, \mu)} 
\leq \|  u_k -   u_j\|_{L^{\widehat q}(\Omega, \mu)} \mu ({\rm supp} (1-\xi
))^{\frac 1q - \frac 1{\widehat q} } 
 \leq C\varepsilon^{\frac 1q - \frac 1{\widehat q} }
\end{align}
for some constant $C$  independent of $k$ and $j$. Inequalities
\eqref{comp1} and \eqref{comp2} imply  that
\begin{equation}\label{comp3}
\|u_k -  u_j\|_{L^q(\Omega , \mu)} \leq \| u_k -  u_j\|_{L^q(\Omega ', \mu)} +  \|(1-\xi
)( u_k-   u_j)\|_{L^q(\Omega, \mu)} \leq
\varepsilon + C\varepsilon^{\frac 1q - \frac 1{\widehat q} }
\end{equation}
if $k, j > k_0$. Owing to the arbitrariness of $\varepsilon$,
inequality \eqref{comp3} tells us that $\{ u_k\}$ is a
Cauchy sequence in $L^{q}(\Omega , \mu)$. 
\\ Part (ii) Given   $\varepsilon >0$, let $K$, $\Omega$, $\xi$ as above, and let $\{u_k\}$ be a bounded sequence in the space 
$ {V^{2}L^{\frac n2}(\Omega ) \cap \mathcal V^{1}L^s(\partial \Omega) \cap C _{\rm b}(\overline
\Omega )}$. By Theorem \ref{fried2_p<n/2},  with $\mu = \mathcal
L^n$, the sequence  $\{u_k\}$ is   bounded in  
$W^{{2},\frac n2}(\Omega ')$ as well. From \cite[Theorem 5.3, Part (ii)]{CavMih} we hence deduce that  $\{u_k\}$ is a Cauchy sequence in the space $\exp L^\gamma (\Omega', \mu)$. Note that, as observed in  \cite{CavMih}, the  theorem in question, although stated for the space $W^{{2},\frac n2}_0(\Omega ')$, continues to hold for $W^{{2},\frac n2}(\Omega ')$ if $\Omega'$ is regular enough. Therefore
there exists $k_0 \in \mathbb N$ such that
\begin{equation}\label{compexp1}
\| u_k -  u_j\|_{\exp L^\gamma (\Omega', \mu)}<\varepsilon
\end{equation}
if $k, j > k_0$. Call $\widehat \gamma$  the minimum in equation \eqref{compnew4}. On making use of a  version of H\"older's inequality in Orlicz spaces one can deduce that
\begin{align}\label{compexp2}
\|(1-\xi )  (u_k - u_j)\|_{\exp L^\gamma (\Omega', \mu)} 
& \leq C \|  u_k -   u_j\|_{\exp L^{\widehat \gamma} (\Omega', \mu)}  \log^{\frac 1 \gamma   -\frac 1 {\widehat \gamma}}\Big(1+ \tfrac 1{\mu ({\rm supp} (1-\xi
)) }\Big) \nonumber
  \\   & \leq C'  \log^{\frac 1 \gamma   -\frac 1 {\widehat \gamma}}\big(1+ \tfrac 1 \varepsilon\big)
\end{align}
for some constants $C$ and $C'$  independent of $k$ and $j$.  Owing to the arbitrariness of $\varepsilon$,
inequalities \eqref{compexp1} and \eqref{compexp2} tell us that $\{ u_k\}$ is a
Cauchy sequence in $\exp L^\gamma (\Omega', \mu)$. 
\end{proof}

\noindent

\begin{acknowledgments}
This research was partly supported by:  
 Research Project 2201758MTR2  of the Italian Ministry of University and
Research (MIUR) Prin 2017 ``Direct and inverse problems for partial differential equations: theoretical aspects and applications''; 
 GNAMPA of the Italian INdAM -- National Institute of High Mathematics
(grant number not available); RUDN University Program 5-100
\end{acknowledgments}

\end{document}